\documentclass[10pt,reqno]{amsart}
\usepackage{amsmath}
\usepackage{amssymb}
\usepackage{amsthm}

\newlength{\defbaselineskip}
\newcommand{\setlinespacing}[1]%
           {\setlength{\baselineskip}{#1 \defbaselineskip}}

\numberwithin{equation}{section}

\newtheorem{thm}{Theorem}[section]
\newtheorem{cor}[thm]{Corollary}
\newtheorem{lem}[thm]{Lemma}
\newtheorem{prop}[thm]{Proposition}

\theoremstyle{definition}

\theoremstyle{remark}
\newtheorem{rem}[thm]{Remark}
\numberwithin{equation}{section}

\begin{document}

\title[Global well-posedness for higher-order Schr\"odinger equations]
{Global well-posedness for higher-order Schr\"odinger equations in weighted $L^2$ spaces}

\author{Youngwoo Koh and Ihyeok Seo}

\thanks{The first named author was supported by NRF grant 2012-008373.}

\address{School of Mathematics, Korea Institute for Advanced Study, Seoul 130-722, Republic of Korea}
\email{ywkoh@kias.re.kr}

\address{Department of Mathematics, Sungkyunkwan University, Suwon 440-746, Republic of Korea}
\email{ihseo@skku.edu}

\maketitle


\begin{abstract}
We obtain the global well-posedness for Schr\"odinger equations of higher orders in weighted $L^2$ spaces.
This is based on weighted $L^2$ Strichartz estimates for the corresponding propagator with higher-order dispersion.
Our method is also applied to the Airy equation which is the linear component of Korteweg-de Vries type equations.
\end{abstract}

\section{Introduction}

In this paper we are concerned with global well-posedness of the following Cauchy problem
for perturbed Schr\"odinger equations of higher orders $a>2$:
\begin{equation}\label{higher}
\begin{cases}
i\partial_tu+(-\Delta)^{a/2}u+V(x,t)u=F(x,t),\\
u(x,0)=u_0(x),
\end{cases}
\end{equation}
where $(x,t)\in\mathbb{R}^{n+1}$, $n\geq1$,
and $(-\Delta)^{a/2}$ is given by means of the Fourier transform $\mathcal{F}f$ $(=\widehat{f}\,)$:
$$\mathcal{F}[(-\Delta)^{a/2}f](\xi)=|\xi|^a\widehat{f}(\xi).$$

Nowadays, the time-dependent Schr\"odinger operator $i\partial_tu-\Delta$
became one of the most popular differential operators of modern mathematics as well as physics.
The higher-order counterpart of it has been also attracted for decades from mathematical physics.
For instance, the forth-order ($a=4$) dispersion term $\Delta^2$ has been
used in the formation and propagation of intense laser beams in a bulk medium with a nonlinearity $f$, as follows (\cite{K,KS}):
$$i\partial_tu+\delta\Delta^2u=f(|u|)u,$$
where $\delta>0$ or $\delta<0$.
Global well-posedness of solutions for this PDE model has been studied in Sobolev spaces
(see, for example, \cite{MXZ,P,P2,W}).

Our work here is aimed at finding a suitable condition on the perturbed term $V(x,t)$
which guarantees existence and uniqueness of solutions to \eqref{higher}
in the weighted $L^2$ space, $L^2(|V|dxdt)$, if $u_0\in L^2$ and $F\in L^2(|V|^{-1}dxdt)$.
Furthermore, it turns out that the solution $u$ belongs to $C_tL_x^2$.

We will consider a function class, denoted by $\mathfrak{L}^{\alpha,\beta,p}$, of $V$ to suit our purpose,
which is defined by
$$\|V\|_{\mathfrak{L}^{\alpha,\beta,p}}
:=\sup_{(x,t)\in\mathbb{R}^{n+1},r,l>0}r^\alpha l^\beta
\bigg(\frac{1}{r^nl} \int_{Q(x,r) \times I(t,l)} |V(y,s)|^p dyds \bigg)^{\frac{1}{p}}<\infty$$
for $0<\alpha\leq n/p$ and $0<\beta\leq 1/p$.
Here, $Q(x,r)$ denotes a cube in $\mathbb{R}^n$ centered at $x$ with side length $r$,
and $I(t,l)$ denotes an interval in $\mathbb{R}$ centered at $t$ with length $l$.
From the definition, it is an elementary matter to check that
$\mathfrak{L}^{\alpha,\beta,p}$ has the dilation property
$\|V(\lambda\cdot,\lambda^a\cdot)\|_{\mathfrak{L}^{\alpha,\beta,p}}
=\lambda^{-\alpha}\lambda^{-a\beta}\|V\|_{\mathfrak{L}^{\alpha,\beta,p}}$.
It should be noted that when $r=l$,
the above class is just the same as the usual Morrey-Campanato class
$\mathfrak{L}^{\alpha+\beta,p} (\mathbb{R}^{n+1})$.
Also, when $r=\sqrt{l}$,
it becomes equivalent to the so-called parabolic Morrey-Campanato class
$\mathfrak{L}^{\alpha+2\beta,p}_{par}(\mathbb{R}^{n+1})$ introduced in \cite{BBCRV}.
Furthermore,
$\mathfrak{L}^{\beta,p}(\mathbb{R};\mathfrak{L}^{\alpha,p}(\mathbb{R}^n))
\subset\mathfrak{L}^{\alpha,\beta,p}(\mathbb{R}^{n+1})$.
Recall that $\mathfrak{L}^{\alpha,p}=L^{p}$ when $p=n/\alpha$,
and even $L^{n/\alpha,\infty}\subset\mathfrak{L}^{\alpha,p}$ for $p<n/\alpha$.

Our well-posedness result for the above Cauchy problem is stated as follows:

\begin{thm}\label{thm0}
Let $n\geq1$ and $a>(n+2)/2$.
Assume that $V\in\mathfrak{L}^{\alpha,\beta,p}$ with
$\|V\|_{\mathfrak{L}^{\alpha,\beta,p}}$ small enough
for $1<p<2$, $a=\alpha+a\beta$ and $\alpha+\beta\geq(n+2)/2$.
Then, if $u_0\in L^2$ and $F\in L^2(|V|^{-1})$, there exists a unique solution of the problem \eqref{higher}
in the space $L^2(|V|)$. Furthermore, the solution $u$ belongs to $C_tL_x^2$ and satisfies the following inequalities:
\begin{equation}\label{1-1}
\|u\|_{L^2(|V|)}\leq C\|V\|_{\mathfrak{L}^{\alpha,\beta,p}}^{1/2}\|u_0\|_{L^2}+C\|V\|_{\mathfrak{L}^{\alpha,\beta,p}}\|F\|_{L_{t,x}^2(|V|^{-1})}
\end{equation}
and
\begin{equation}\label{1-2}
\sup_{t\in\mathbb{R}}\|u\|_{L_x^2}\leq C\|u_0\|_{L^2}+C\|V\|_{\mathfrak{L}^{\alpha,\beta,p}}^{1/2}\|F\|_{L_{t,x}^2(|V|^{-1})}.
\end{equation}
\end{thm}

A natural way to achieve this result is to obtain weighted $L^2$ estimates for
the solutions in terms of the initial datum $u_0$ and the forcing term $F$.
To be precise, let us first consider the following non-perturbed problem:
\begin{equation*}
\begin{cases}
i\partial_tu+(-\Delta)^{a/2}u=F(x,t),\\
u(x,0)=f(x).
\end{cases}
\end{equation*}
As is well known, the solution is then given by
\begin{equation*}
u(x,t)=e^{it(-\Delta)^{a/2}}f(x)-i\int_0^t e^{i(t-s)(-\Delta)^{a/2}}F(\cdot,s)ds,
\end{equation*}
where the evolution operator $e^{it(-\Delta)^{a/2}}$ is defined by
$$e^{it(-\Delta)^{a/2}}f(x)=\int_{\mathbb{R}^n}e^{ix\cdot\xi}e^{it|\xi|^a}\widehat{f}(\xi)d\xi.$$
The proof of Theorem \ref{thm0} will be done in the next section
by making use of the following weighted $L^2$ Strichartz estimates:

\begin{thm}\label{thm}
Let $n\geq1$ and $a>(n+2)/2$. Assume that $w\geq0$ is a function in $\mathfrak{L}^{\alpha,\beta,p}$. Then we have
\begin{equation}\label{hop}
\big\|e^{it(-\Delta)^{a/2}}f \big\|_{L^2(w(x,t))}\leq C\|w\|_{\mathfrak{L}^{\alpha,\beta,p}}^{1/2}\|f\|_{L^2},
\end{equation}
and
\begin{equation}\label{inho}
\bigg\|\int_{0}^{t}e^{i(t-s)(-\Delta)^{a/2}}F(\cdot,s)ds\bigg\|_{L^2(w(x,t))}
\leq C\|w\|_{\mathfrak{L}^{\alpha,\beta,p}} \|F\|_{L^2(w(x,t)^{-1})}
\end{equation}
if
\begin{equation}\label{cond}
1<p<2,\quad a=\alpha+a\beta\quad\text{and}\quad\alpha+\beta\geq(n+2)/2.
\end{equation}
\end{thm}

\begin{rem}
The condition $a=\alpha+a\beta$ in the theorem is needed for the scaling invariance of the estimates \eqref{hop} and \eqref{inho}
under the scaling $(x,t)\rightarrow(\lambda x,\lambda^at)$, $\lambda>0$.
\end{rem}

\begin{rem}\label{rem}
Since $0<\alpha\leq n/p$ and $0<\beta\leq 1/p$,
it is not difficult to check that for $n\geq1$ and $a>(n+2)/2$, there exist $\alpha$, $\beta$ and $p$ satisfying the condition \eqref{cond} as follows:
$$1<p\leq\min\{\frac{2(a-1)}{a},\frac{a+n}a\},$$
$$\max\{\frac{an}{2(a-1)},\frac{a}{p'}\}\leq\alpha\leq\frac np,$$
$$1-\frac{n}{ap} \leq \beta \leq \min\{ \frac{n+2-2a}{2(1-a)}, \frac{1}{p} \}.$$
At this point, we would like to point out that when $n=1$, our theorems can also cover the fractional order cases where $3/2<a<2$.
\end{rem}

The Strichartz estimates for the Schr\"odinger equation ($a=2$) in the usual $L_t^qL_x^r$ norms
have been extensively developed in the works \cite{Str,GV,CW,Ka,KT,F,V,Ko,LS}.
See also \cite{CN,CN2,S,LS2} for related results.
In the weighted $L^2$ setting as above, they were also studied in \cite{RV,BBRV,S3} using weighted $L^2$ resolvent estimates for the Laplacian,
and were applied to the problem of well-posedness for the Schr\"odinger equation in weighted $L^2$ spaces.
Our method here for higher orders $a>2$ is entirely different from them and is inspired by \cite{BBCRV} and our earlier work \cite{KoS} for the wave operator $-\partial_t^2+\Delta$.
It will be based on a combination of a localization argument in weighted $L^2$ spaces and a bilinear interpolation argument. See the final section, Section \ref{sec4}.
The key ingredient in doing so is Lemma \ref{lem2} in Section \ref{sec3} which enables us to make use of
frequency localized estimates (see Proposition \ref{prop}) in weighted $L^2$ spaces
without additional assumption on the weight $w$.

Let us now mention an implication of Theorem \ref{thm} for the Airy equation
which is the linear component of Korteweg-de Vries type equations.
First, consider the corresponding problem
\begin{equation}\label{airy}
\begin{cases}
\partial_tu+\partial_x^3u=F(x,t),\\
u(x,0)=f(x),
\end{cases}
\end{equation}
where $(x,t)\in\mathbb{R}^{1+1}$.
Using the Fourier transform, the solution is then written as
\begin{equation*}
u(x,t)=e^{-t\partial_x^3}f(x)+\int_0^t e^{-(t-s)\partial_x^3}F(\cdot,s)ds.
\end{equation*}
Here we note that the operator $e^{-t\partial_x^3}$ is given by
$$e^{-t\partial_x^3}f(x)=\int_{\mathbb{R}^1}e^{ix\xi}e^{it\xi^3}\widehat{f}(\xi)d\xi.$$
Thus, the proof of Theorem \ref{thm} when $n=1$ and $a=3$ would be clearly worked for \eqref{airy}.
Then, from this observation and Remark \ref{rem}, the resulting estimates would be given as follows:

\begin{thm}
Assume that $w\geq0$ is a function in $\mathfrak{L}^{\alpha,\beta,p}(\mathbb{R})$. Then we have
\begin{equation*}
\big\|e^{-t\partial_x^3}f \big\|_{L^2(w(x,t))}\leq C\|w\|_{\mathfrak{L}^{\alpha,\beta,p}}^{1/2}\|f\|_{L^2},
\end{equation*}
and
\begin{equation*}
\bigg\|\int_{0}^{t}e^{-(t-s)\partial_x^3}F(\cdot,s)ds\bigg\|_{L^2(w(x,t))}
\leq C\|w\|_{\mathfrak{L}^{\alpha,\beta,p}} \|F\|_{L^2(w(x,t)^{-1})}
\end{equation*}
for $1<p\leq4/3$, $3/4\leq\alpha\leq1/p$ and $1-1/(3p)\leq\beta\leq3/4$.
\end{thm}

Also, from the same argument for Theorem \ref{thm0}, this theorem implies the following global well-posedness for the corresponding Cauchy problem
\begin{equation}\label{airy2}
\begin{cases}
\partial_tu+\partial_x^3u+V(x,t)u=F(x,t),\\
u(x,0)=u_0(x).
\end{cases}
\end{equation}

\begin{cor}
Assume that $V\in\mathfrak{L}^{\alpha,\beta,p}$ with
$\|V\|_{\mathfrak{L}^{\alpha,\beta,p}}$ small enough
for  $1<p\leq4/3$, $3/4\leq\alpha\leq1/p$ and $1-1/(3p)\leq\beta\leq3/4$.
Then, if $u_0\in L^2$ and $F\in L^2(|V|^{-1})$, there exists a unique solution of the problem \eqref{airy2}
in the space $L^2(|V|)$. Furthermore, the solution $u$ belongs to $C_tL_x^2$ and satisfies the inequalities \eqref{1-1} and \eqref{1-2}.
\end{cor}

Throughout this paper, the letter $C$ stands for a constant which may be different
at each occurrence. Also, we denote by $\widehat{f}$ the Fourier transform of $f$
and by $\langle f,g\rangle$ the usual inner product of $f,g$ on $L^2$.


\section{Proof of Theorem \ref{thm0}}

In this section we explain how to deduce the well-posedness (Theorem \ref{thm0})
for the Cauchy problem \eqref{higher} from the weighted $L^2$ Strichartz estimates in Theorem \ref{thm}.

The starting point is that the solution of \eqref{higher} is given by the following integral equation
\begin{equation}\label{sol}
u(x,t)= e^{it(-\Delta)^{a/2}}u_0(x) - i\int_0^t e^{i(t-s)(-\Delta)^{a/2}} F(\cdot,s)ds
+\Phi(u)(x,t),
\end{equation}
where
$$\Phi(u)(x,t)=-i\int_0^t e^{i(t-s)(-\Delta)^{a/2}}(Vu)(\cdot,s)ds.$$
Note here that
$$(I-\Phi)(u)=e^{it(-\Delta)^{a/2}}u_0(x) - i\int_0^t e^{i(t-s)(-\Delta)^{a/2}} F(\cdot,s)ds,$$
where $I$ is the identity operator.
Then, since $u_0\in L^2$ and $F\in L^2(|V|^{-1})$,
from the weighted $L^2$ Strichartz estimates in Theorem \ref{thm} with $w=|V|$,
it follows that
$$(I-\Phi)(u)\in L^2(|V|).$$
Hence, for the global existence of the solution, we want to show that
the operator $I-\Phi$ has an inverse in the space $L^2(|V|)$.
This is valid if the operator norm for $\Phi$ in $L^2(|V|)$ is strictly less than $1$.
Namely, we need to show $\|\Phi(u)\|_{L^2(|V|)}<\|u\|_{L^2(|V|)}$.
But, from the inhomogeneous estimate \eqref{inho} with $w=|V|$, we see that
\begin{align}\label{inv}
\nonumber\|\Phi(u)\|_{L^2(|V|)}&\leq C\|V\|_{\mathfrak{L}^{\alpha,\beta,p}}\|Vu\|_{L^2(|V|^{-1})}\\
\nonumber&=C\|V\|_{\mathfrak{L}^{\alpha,\beta,p}}\|u\|_{L^2(|V|)}\\
&<\frac12\|u\|_{L^2(|V|)}
\end{align}
by the smallness assumption on the norm $\|V\|_{\mathfrak{L}^{\alpha,\beta,p}}$.

On the other hand, from \eqref{sol}, \eqref{inv} and Theorem \ref{thm}, it follows easily that
\begin{align}\label{1123}
\nonumber\|u\|_{L^2(|V|)}&\leq C\big\|e^{it(-\Delta)^{a/2}}u_0 \big\|_{L^2(|V|)}
+C\bigg\|\int_{0}^{t}e^{i(t-s)(-\Delta)^{a/2}}F(\cdot,s)ds\bigg\|_{L^2(|V|)}\\
&\leq C\|V\|_{\mathfrak{L}^{\alpha,\beta,p}}^{1/2}\|u_0\|_{L^2}+C\|V\|_{\mathfrak{L}^{\alpha,\beta,p}}\|F\|_{L_{t,x}^2(|V|^{-1})}.
\end{align}
Hence \eqref{1-1} is proved.
To show \eqref{1-2}, we will use \eqref{1123} and the following estimate
\begin{equation}\label{dual}
    \bigg\|\int_{-\infty}^{\infty}e^{-is(-\Delta)^{a/2}}F(\cdot,s)ds\bigg\|_{L_x^2}
    \leq C\|w\|_{\mathfrak{L}^{\alpha,\beta,p}}^{1/2}\|F\|_{L^2(w(x,t)^{-1})}
\end{equation}
which is the dual version of \eqref{hop}.
First, from \eqref{sol}, \eqref{dual} with $w=|V|$, and the simple fact that $e^{it(-\Delta)^{a/2}}$ is an isometry in $L^2$,
one can see that
$$\|u\|_{L_x^2}\leq C\|u_0\|_{L^2}
+C\|V\|_{\mathfrak{L}^{\alpha,\beta,p}}^{1/2}\|F\|_{L^2(|V|^{-1})}
+C\|V\|_{\mathfrak{L}^{\alpha,\beta,p}}^{1/2}\|Vu\|_{L^2(|V|^{-1})}.$$
Since $\|Vu\|_{L^2(|V|^{-1})}=\|u\|_{L^2(|V|)}$ and $\|V\|_{\mathfrak{L}^{\alpha,\beta,p}}$ is small,
from this and \eqref{1123}, it follows now that
$$\|u\|_{L_x^2}\leq C\|u_0\|_{L^2}+C\|V\|_{\mathfrak{L}^{\alpha,\beta,p}}^{1/2}\|F\|_{L_{t,x}^2(|V|^{-1})}.$$
So we get \eqref{1-2}.
This completes the proof.

\section{Preliminaries}\label{sec3}

Here we present some preliminary lemmas which are needed in the next section for the proof of Theorem \ref{thm}.

Given two complex Banach spaces $A_0$ and $A_1$,
for $0<\theta<1$ and $1\leq q\leq\infty$, we denote by $(A_0,A_1)_{\theta,q}$
the real interpolation spaces equipped with the norms
$$\|a\|_{(A_0,A_1)_{\theta,\infty}}=\sup_{0<t<\infty}t^{-\theta}K(t,a)$$
and
$$\|a\|_{(A_0,A_1)_{\theta,q}}=\bigg(\int_0^\infty(t^{-\theta}K(t,a))^qdt\bigg)^{1/q},\quad1\leq q<\infty,$$
where
$$K(t,a)=\inf_{a=a_0+a_1}\|a_0\|_{A_0}+t\|a_1\|_{A_1}$$
for $0<t<\infty$ and $a\in A_0+A_1$.
In particular, $(A_0,A_1)_{\theta,q}=A_0=A_1$ if $A_0=A_1$.
See ~\cite{BL,T} for details.
The following bilinear interpolation lemma concerning these interpolation spaces
is well-known (see \cite{BL}, Section 3.13, Exercise 5(b)).

\begin{lem}\label{lem}
For $i=0,1$, let $A_i,B_i,C_i$ be Banach spaces and let $T$ be a bilinear operator such that
\begin{align*}
&T:A_0\times B_0\rightarrow C_0,\\
&T:A_0\times B_1\rightarrow C_1,\\
&T:A_1\times B_0\rightarrow C_1.
\end{align*}
Then one has for $\theta=\theta_0+\theta_1$ and $1/q+1/r\geq1$
$$T:(A_0,A_1)_{\theta_0,q}\times(B_0,B_1)_{\theta_1,r}\rightarrow(C_0,C_1)_{\theta,1}.$$
Here\, $0<\theta_i<\theta<1$ and $1\leq q,r\leq\infty$.
\end{lem}

Let us now recall that a weight\footnote{\,It is a locally integrable function
which is allowed to be zero or infinite only on a set of Lebesgue measure zero.}
$w:\mathbb{R}^n\rightarrow[0,\infty]$
is said to be in the Muckenhoupt $A_2(\mathbb{R}^n)$ class if there is a constant $C_{A_2}$ such that
\begin{equation*}
\sup_{Q\text{ cubes in }\mathbb{R}^{n}}
\bigg(\frac1{|Q|}\int_{Q}w(x)dx\bigg)\bigg(\frac1{|Q|}\int_{Q}w(x)^{-1}dx\bigg)<C_{A_2}.
\end{equation*}
(See, for example, \cite{G}.)
In the following we obtain a useful property of a weight in the space $\mathfrak{L}^{\alpha,\beta,p}$.
For some similar properties of the usual Morrey-Campanato spaces, we refer the reader to \cite{CS,KoS}.
Such property has been used earlier in \cite{CS,S2,S3} concerning unique continuation for Schr\"odinger equations.

\begin{lem}\label{lem2}
For a weight $w\in\mathfrak{L}^{\alpha,\beta,p}$ on $\mathbb{R}^{n+1}$,
let $w_*(x,t)$ be the $n$-dimensional maximal function given by
$$w_*(x,t)=\sup_{Q'}\bigg(\frac{1}{|Q'|}\int_{Q'}w(y,t)^\rho dy\bigg)^{\frac{1}{\rho}},\quad\rho>1,$$
where $Q'$ denotes a cube in $\mathbb{R}^n$ with center $x$.
Then, if $p>\rho$
we have $\|w_*\|_{\mathfrak{L}^{\alpha,\beta,p}}\leq C\|w\|_{\mathfrak{L}^{\alpha,\beta,p}}$,
and $w_\ast(\cdot,t)\in A_2(\mathbb{R}^n)$ in the $x$ variable with a constant $C_{A_2}$ uniform in almost every $t\in\mathbb{R}$.
\end{lem}

\begin{proof}
To show $\|w_*\|_{\mathfrak{L}^{\alpha,\beta,p}}\leq C\|w\|_{\mathfrak{L}^{\alpha,\beta,p}}$,
we first fix a cube $Q(z,r)\times I(\tau,l)$ in $\mathbb{R}^{n+1}$.
Then, we define the rectangles $R_k$, $k\geq1$, such that $(y,t)\in R_k$ if $|t-\tau|<2l$
and $y\in Q(z,2^{k+1}r)\setminus Q(z,2^kr)$,
and set $R_0=Q(z,4r)\times I(\tau,4l)$.

Now we can write
$$w(y,t)=\sum_{k\geq 0}w^{(k)}(y,t)+\phi(y,t),$$
where $w^{(k)}=w\chi_{R_k}$ with the characteristic function $\chi_{R_k}$ of the set $R_k$,
and $\phi(y,t)$ is a function supported on $\mathbb{R}^{n+1}\setminus\bigcup_{k\geq0}R_k$.
Also it is not difficult to see that
$$w_*(x,t)\leq\sum_{k \geq 0}w^{(k)}_*(x,t)+\phi_*(x,t)$$
and
\begin{align*}
\Big(\int_{Q(z,r)\times I(\tau,l)}w_*(x,t)^pdxdt\Big)^{\frac{1}{p}}
&\leq\sum_{k \geq 0}\Big(\int_{Q(z,r)\times I(\tau,l)} w^{(k)}_*(x,t)^p dxdt \Big)^{\frac{1}{p}}\\
&+\Big( \int_{Q(z,r)\times I(\tau,l)}\phi_*(x,t)^pdxdt\Big)^{\frac{1}{p}}.
\end{align*}
Since $(x,t)\in Q(z,r) \times Q(\tau,l)$, it is clear that $\phi_*(x,t)=0$.
For the term where $k=0$, we use the following well-known maximal theorem
$$\|M(f)\|_q\leq C\|f\|_q,\quad q>1,$$
where $M(f)$ is the usual Hardy-Littlewood maximal function defined by
\begin{equation}\label{maxi}
M(f)(x)=\sup_{Q}\frac1{|Q|}\int_{Q}f(y)dy.
\end{equation}
(Here, the sup is taken over all cubes $Q$ in $\mathbb{R}^{n}$ with center $x$.)
Indeed, by applying the maximal theorem with $q=p/\rho$ in $x$-variable,
one can see that if $p>\rho$
\begin{align}\label{sso}
\nonumber r^\alpha l^\beta \Big( \frac{1}{r^n l} \int_{Q(z,r) \times I(\tau,l)}&w^{(0)}_* (x,t)^p dxdt \Big)^{\frac{1}{p}}\\
\nonumber\leq &C r^\alpha l^\beta \Big( \frac{1}{r^n l} \int_{Q(z,4r) \times I(\tau,4l)} w (y,t)^p dydt \Big)^{\frac{1}{p}}\\
\leq &C\|w\|_{\mathfrak{L}^{\alpha,\beta,p}}.
\end{align}
So it remains to consider the term where $k\geq1$.

Let $k \geq 1$.
Since $(x,t)\in Q(z,r) \times Q(\tau,l)$, it follows that
    \begin{align*}
    w^{(k)}_*(x,t)
    &= \sup_{Q' \subset \mathbb{R}^n}\bigg(\frac{1}{|Q'|}\int_{Q'} w(y,t)^\rho\chi_{R_k}(y,t)dy\bigg)^{\frac{1}{\rho}}\\
    &\leq C\bigg(\frac{1}{(2^k r)^n}\int_{Q(z,2^{k+1} r)\setminus Q(z, 2^k r)} w(y,t)^\rho dy \bigg)^{\frac{1}{\rho}}\\
    &\leq C\bigg(\frac{1}{(2^k r)^n}
    \int_{Q(z,2^{k+1}r)\setminus Q(z, 2^k r)} w(y,t)^pdy\bigg)^{\frac{1}{p}},
    \end{align*}
where we used H\"older's inequality for the last inequality since $p\geq\rho$.
Hence,
\begin{align*}
\int_{Q(z,r) \times I(\tau,l)}& w^{(k)}_*(x,t)^pdxdt \\
\leq &\frac{C}{(2^kr)^n}\int_{|\tau-t|< l}
\int_{Q(z,2^{k+1}r)\setminus Q(z, 2^k r)}w(y,t)^p\int_{|z-x|<r}1\,dxdydt\\
\leq&\frac{C}{2^{kn}}\int_{R_k}w(y,t)^pdydt.
\end{align*}
Since $R_k\subset Q(z,2^{k+1}r)\times I(\tau,2l)$, this implies that
    \begin{align*}
    r^\alpha l^\beta \Big( \frac{1}{r^n l}& \int_{Q(z,r) \times I(\tau,l)} w^{(k)}_* (x,t)^p dxdt \Big)^{\frac{1}{p}} \\
    &\leq C r^\alpha l^\beta \Big( \frac{1}{2^{kn} r^n l}\int_{R_k} w(y,t)^pdydt \Big)^{\frac{1}{p}} \\
    &\leq C 2^{-\alpha k} (2^k r)^\alpha l^\beta \Big( \frac{1}{(2^{k}r)^nl} \int_{Q(z,2^{k+1}r)\times I(\tau,2l)} w(y,t)^p dy dt \Big)^{\frac{1}{p}} \\
    &\leq C 2^{-\alpha k} \|w\|_{\mathfrak{L}^{\alpha,\beta,p}} .
    \end{align*}
Hence, since $\alpha >0$ and $p\geq\rho$, it follows that
    $$
     \sum_{k \geq 1}r^\alpha l^\beta \Big( \frac{1}{r^n l} \int_{Q(z,r) \times Q(\tau,l)} w_*^{(k)}(x,t)^pdxdt\Big)^{\frac{1}{p}}
     \leq C\|w\|_{\mathfrak{L}^{\alpha,\beta,p}}.
    $$
Consequently, combining this and \eqref{sso}, we get
$\|w_*\|_{\mathfrak{L}^{\alpha,\beta,p}}\leq C\|w\|_{\mathfrak{L}^{\alpha,\beta,p}}$
if $p>\rho$, as desired.

Next, for the remaining part ($w_\ast(\cdot,t)\in A_2(\mathbb{R}^n)$) in the lemma, we first need to recall some known facts for $A_1$ weights.
We say that $w$ is in the class $A_1$ if there is a constant $C_{A_1}$
such that for almost every $x$
\begin{equation*}
M(w)(x)\leq C_{A_1}w(x),
\end{equation*}
where $M(w)$ is the Hardy-Littlewood maximal function of $w$ (see \eqref{maxi}).
Then,
\begin{equation}\label{bac}
A_1\subset A_2\quad\text{with}\quad C_{A_2}\leq C_{A_1}.
\end{equation}
See, for example, \cite{G} for details.
Also, the following fact\footnote{\,It can be found in Chapter 5 of \cite{St}.
See also Proposition 2 in \cite{CR}.} is known:
If $M(w)(x)<\infty$ for almost every $x\in\mathbb{R}^n$, then for $0<\delta<1$
\begin{equation}\label{max}
(M(w))^\delta\in A_1
\end{equation}
with $C_{A_1}$ independent of $w$.

Now we are ready to show that
$w_\ast(\cdot,t)\in A_2(\mathbb{R}^n)$ in the $x$ variable with a constant $C_{A_2}$ uniform in almost every $t\in\mathbb{R}$.
Note first that
$$w_\ast(x,t)=(M(w(\cdot,t)^\rho))^{1/\rho}.$$
Since $w\in\mathfrak{L}^{\alpha,\beta,p}$ and $p\geq\rho$, it is an elementary matter to check that
$M(w(\cdot,t)^\rho)<\infty$ for almost every $x\in\mathbb{R}^n$.
Then, by applying \eqref{max} with $\delta=1/\rho$,
it follows that $w_\ast(\cdot,t)\in A_1$ with $C_{A_1}$ uniform in $t\in\mathbb{R}$.
Finally, from \eqref{bac} this immediately implies that
$w_\ast(\cdot,t)\in A_2$ with $C_{A_2}$ uniform in $t\in\mathbb{R}$.
\end{proof}

\section{Proof of Theorem \ref{thm}}\label{sec4}

Since $w\leq w_\ast$ and
$\|w_*\|_{\mathfrak{L}^{\alpha,\beta,p}}\leq C\|w\|_{\mathfrak{L}^{\alpha,\beta,p}}$ for $p>\rho>1$
(see Lemma \ref{lem2}), if we show the homogeneous estimate \eqref{hop} replacing $w$ with $w_\ast$, we get
    \begin{align*}
    \big\|e^{it(-\Delta)^{a/2}} f \big\|_{L^2(w(x,t))}
    &\leq \big\|e^{it(-\Delta)^{a/2}} f \big\|_{L^2(w_*(x,t))} \\
    &\leq C\|w_*\|_{\mathfrak{L}^{\alpha,\beta,p}}^{1/2}\|f\|_{L^2}\\
    &\leq C\|w\|_{\mathfrak{L}^{\alpha,\beta,p}}^{1/2}\|f\|_{L^2}
    \end{align*}
as desired.
Similarly for the inhomogeneous estimate \eqref{inho}.
So it suffices to prove Theorem \ref{thm} by replacing $w$ with $w_\ast$.
From this replacement we are in a good light that we can use the property $w_\ast(\cdot,t)\in A_2(\mathbb{R}^n)$ in Lemma \ref{lem2}.
In fact, this $A_2$ condition will enable us to make use of a localization argument in weighted $L^2$ spaces.
From the above argument, we shall assume, for simplicity of notation, that $w$ satisfies the same $A_2$ condition.

Now, let $\phi$ be a smooth function supported in $(1/2,2)$ such that
$$\sum_{k=-\infty}^\infty \phi(2^k t)=1,\quad t>0.$$
Then we define the multiplier operators $P_kf$ for $k\in\mathbb{Z}$ by
$$\widehat{P_kf}(\xi)=\phi(2^{-k}|\xi|)\widehat{f}(\xi).$$
Then we have the following.

\begin{prop}\label{prop}
If $\alpha+\beta\geq(n+2)/2$ and $1<p<2$, then
\begin{equation}\label{freq}
\big\|e^{it(-\Delta)^{a/2}} P_k f \big\|_{L^2(w(x,t))}\leq C2^{k(\alpha + a\beta -a)/2}
\|w\|_{\mathfrak{L}^{\alpha,\beta,p}}^{1/2}\|f\|_{L^2}
\end{equation}
and
\begin{equation}\label{base}
\bigg\|\int_{0}^{t}e^{i(t-s)(-\Delta)^{a/2}}P_k F(\cdot,s)ds\bigg\|_{L^2(w(x,t))}
\leq C2^{k(\alpha + a\beta -a)/2}\|w\|_{\mathfrak{L}^{\alpha,\beta,p}} \|F\|_{L^2(w(x,t)^{-1})}.
\end{equation}
\end{prop}

Assuming this proposition for the moment, we prove Theorem \ref{thm}.
First we consider the homogeneous estimate.
As mentioned above, since we may assume that
$w(\cdot,t)\in A_2(\mathbb{R}^{n})$ uniformly for almost every $t \in \mathbb{R}$,
by the Littlewood-Paley theorem on weighted $L^2$ spaces (see Theorem 1 in \cite{Ku}), we see that
    \begin{align*}
    \big\|e^{it(-\Delta)^{a/2}}f\big\|_{L^2(w(x,t))}^2
    &=\int\big\|e^{it(-\Delta)^{a/2}}f\big\|_{L^2(w(\cdot,t))}^2dt\\
    &\leq C\int\bigg\|\bigg(\sum_k\big|P_ke^{it(-\Delta)^{a/2}}f\big|^2\bigg)^{1/2}\bigg\|_{L^2(w(\cdot,t))}^2dt\\
    &=C\sum_k\big\|e^{it(-\Delta)^{a/2}}P_kf\big\|_{L^2(w(x,t))}^2.
    \end{align*}
Meanwhile, since $P_kP_jf=0$ if $|j-k|\geq2$, it follows from \eqref{freq} that
    \begin{align*}
    \sum_k\big\|e^{it(-\Delta)^{a/2}}P_kf\big\|_{L^2(w(x,t))}^2
    &=\sum_k\big\|e^{it(-\Delta)^{a/2}}P_k\big(\sum_{|j-k|\leq1}P_jf\big)\big\|_{L^2(w(x,t))}^2\\
    &\leq C\|w\|_{\mathfrak{L}^{\alpha,\beta,p}}\sum_k2^{k(\alpha + a\beta -a)}\big\|\sum_{|j-k|\leq1}P_jf\big\|_2^2\\
    &\leq C\|w\|_{\mathfrak{L}^{\alpha,\beta,p}}\|f\|_2^2.
    \end{align*}
Here, recall that $a= \alpha + a\beta$.
Thus, we get
    $$
    \big\|e^{it(-\Delta)^{a/2}} f \big\|_{L^2(w(x,t))}\leq C\|w\|_{\mathfrak{L}^{\alpha,\beta,p}}^{1/2}\|f\|_{L^2}
    $$
as desired.
The inhomogeneous estimate \eqref{inho} follows also from the same argument.
Indeed, by the Littlewood-Paley theorem on weighted $L^2$ spaces as before,
one can see that
    \begin{align*}
    &\bigg\|\int_{0}^{t} e^{i(t-s)(-\Delta)^{a/2}}F(\cdot,s)ds\bigg\|_{L^2(w(x,t))}^2\\
    &\qquad\leq C\sum_{k}\bigg\|\int_{0}^{t}e^{i(t-s)(-\Delta)^{a/2}}P_k\big(\sum_{|j-k|\leq1}P_j F(\cdot,s)\big)ds \bigg\|_{L^2(w(x,t))}^2.
    \end{align*}
By using \eqref{base}, the right-hand side in the above is bounded by
    $$
    C\|w\|_{\mathfrak{L}^{\alpha,\beta,p}}^2
    \sum_{k} 2^{k(\alpha + a\beta -a)} \big\|\sum_{|j-k|\leq1}P_j F\big\|_{L^2(w(x,t)^{-1})}^2 .
    $$
Since $a=\alpha+a\beta$, and $w(\cdot,t)^{-1}\in A_2(\mathbb{R}^n)$ if and only if $w(\cdot,t)\in A_2(\mathbb{R}^n)$,
applying the Littlewood-Paley theorem again, this is bounded by
$C\|w\|_{\mathfrak{L}^{\alpha,\beta,p}}^2\|F\|_{L^2(w(x,t)^{-1})}^2$.
Consequently,
$$\bigg\|\int_{0}^{t}e^{i(t-s)(-\Delta)^{a/2}} F(\cdot,s)ds\bigg\|_{L^2(w(x,t))}
\leq C\|w\|_{\mathfrak{L}^{\alpha,\beta,p}} \|F\|_{L^2(w(x,t)^{-1})}.$$

Now we are reduced to showing the proposition, and so the rest of this section will be devoted to it.

\begin{proof}[Proof of Proposition \ref{prop}]
We first show the estimate \eqref{freq}.
From scaling, it suffices to show the following case where $k=0$:
\begin{equation}\label{freq2}
\big\|e^{it(-\Delta)^{a/2}} P_0 f \big\|_{L^2(w(x,t))}
\leq C\|w\|_{\mathfrak{L}^{\alpha,\beta,p}}^{1/2}\|f\|_{L^2}.
\end{equation}
In fact, note that
    \begin{align*}
    \big\|e^{it(-\Delta)^{a/2}}P_k f\big\|_{L^2(w(x,t))}^2
    &\leq C2^{-kn}2^{-ak}\big\|e^{it(-\Delta)^{a/2}}P_0(f(2^{-k}\cdot))\big\|_{L^2(w(2^{-k}x,2^{-ak}t))}^2\\
    &\leq C2^{-kn}2^{-ak}\|w(2^{-k}x,2^{-ak}t)\|_{\mathfrak{L}^{\alpha,\beta,p}}\|f(2^{-k}\cdot)\|_2^2\\
    &\leq C2^{k(\alpha + a\beta -a)}\|w\|_{\mathfrak{L}^{\alpha,\beta,p}}\|f\|_2^2.
    \end{align*}
Now, by duality, \eqref{freq2} is equivalent to
    \begin{equation*}
    \bigg\|\int e^{-is(-\Delta)^{a/2}}P_0F(\cdot,s)ds\bigg\|_{L_x^2}
    \leq C\|w\|_{\mathfrak{L}^{\alpha,\beta,p}}^{1/2}\|F\|_{L^2(w^{-1})},
    \end{equation*}
and so it is enough to show the following bilinear form estimate
    \begin{equation*}
    \bigg|\bigg\langle\int_{\mathbb{R}}e^{i(t-s)(-\Delta)^{a/2}}P_0^2F(\cdot,s)ds,G(x,t)\bigg\rangle\bigg|
    \leq C\|w\|_{\mathfrak{L}^{\alpha,\beta,p}}\|F\|_{L^2(w^{-1})}\|G\|_{L^2(w^{-1})}.
    \end{equation*}
For this, let us write
    $$\int_{\mathbb{R}}e^{i(t-s)(-\Delta)^{a/2}}P_0^2F(\cdot,s)ds=K\ast F,$$
where
    $$K (x,t)=\int_{\mathbb{R}^n}e^{i(x\cdot\xi+t|\xi|^a)}\phi(|\xi|)^2d\xi.$$
Let $\psi_j:\mathbb{R}^{n+1}\rightarrow[0,1]$ be a smooth function
which is supported in $B(0,2^j) \setminus B(0,2^{j-2})$ for $j\geq1$ and in $B(0,1)$ for $j=0$,
such that $\sum_{j\geq0}\psi_j=1$.
Then, we decompose the kernel $K$ into
$$K=\sum_{j\geq0}\psi_jK,$$
and will show that
    \begin{equation}\label{345}
    \sum_{j\geq0}\big| \big\langle (\psi_j K)\ast F,G\big\rangle \big|
    \leq C\|w\|_{\mathfrak{L}^{\alpha,\beta,p}}\|F\|_{L^2(w^{-1})}\|G\|_{L^2(w^{-1})}
    \end{equation}
which implies the above bilinear form estimate.

To show \eqref{345}, we will obtain the following three estimates
    \begin{align}\label{21}
    \big| \big\langle (\psi_j K) \ast F, G \big\rangle \big|
    &\leq C2^{j (\frac{n+2}{2} -(\alpha+\beta)p)}\|w\|_{\mathfrak{L}^{\alpha,\beta,p}}^p \|F\|_{L^2(w^{-p})}\|G\|_{L^2(w^{-p})},\\
    \label{22}\big| \big\langle (\psi_j K) \ast F, G \big\rangle \big|
    &\leq C2^{j (\frac{n+2}{2} -\frac{\alpha+\beta}{2}p)}\|w\|_{\mathfrak{L}^{\alpha,\beta,p}}^{p/2}\|F\|_{L^2(w^{-p})}\|G\|_2,\\
    \label{23}\big| \big\langle (\psi_j K) \ast F, G \big\rangle \big|
    &\leq C2^{j (\frac{n+2}{2} -\frac{\alpha+\beta}{2}p)}\|w\|_{\mathfrak{L}^{\alpha,\beta,p}}^{p/2}\|F\|_2\|G\|_{L^2(w^{-p})}.
    \end{align}
Assuming these estimates for the moment, let us show the estimate \eqref{345}
by making use of the bilinear interpolation lemma, Lemma \ref{lem}.
First, define the bilinear vector-valued operator $T$ by
\begin{equation*}
T(F,G)=\big\{\big\langle (\psi_j K)\ast F,G\big\rangle\big\}_{j\geq0}.
\end{equation*}
Then \eqref{345} is equivalent to
\begin{equation}\label{bii}
T:L^2(w^{-1})\times L^2(w^{-1})\rightarrow\ell_1^0(\mathbb{C})
\end{equation}
with the operator norm  $C\|w\|_{\mathfrak{L}^{\alpha,\beta,p}}$.
Here, for $a\in\mathbb{R}$ and $1\leq p\leq\infty$,
$\ell^a_p(\mathbb{C})$ denotes the weighted sequence space with the norm
$$\|\{x_j\}_{j\geq0} \|_{\ell^a_p} =
\begin{cases}
\big(\sum_{j\geq0}2^{jap}|x_j|^p\big)^{\frac{1}{p}},
\quad\text{if}\quad p\neq\infty,\\
\,\sup_{j\geq0}2^{ja}|x_j|,
\quad\text{if}\quad p=\infty.
\end{cases}$$
Next, note that the above three estimates \eqref{21}, \eqref{22} and \eqref{23} become
    \begin{align}
    \nonumber\|T(F,G)\|_{\ell_\infty^{\beta_0}(\mathbb{C})}
    &\leq C\|w\|_{\mathfrak{L}^{\alpha,\beta,p}}^p \|F\|_{L^2(w^{-p})}\|G\|_{L^2(w^{-p})},\\
    \label{222}\|T(F,G)\|_{\ell_\infty^{\beta_1}(\mathbb{C})}
    &\leq C\|w\|_{\mathfrak{L}^{\alpha,\beta,p}}^{p/2}\|F\|_{L^2(w^{-p})}\|G\|_2,\\
    \label{233}\|T(F,G)\|_{\ell_\infty^{\beta_1}(\mathbb{C})}
    &\leq C\|w\|_{\mathfrak{L}^{\alpha,\beta,p}}^{p/2}\|F\|_2\|G\|_{L^2(w^{-p})},
    \end{align}
respectively, with $\beta_0=-(\frac{n+2}{2} -(\alpha+\beta)p)$
and $\beta_1=-(\frac{n+2}{2}-\frac{\alpha+\beta}{2}p)$.
Then, applying Lemma \ref{lem} with $\theta_0=\theta_1=1/p^\prime$ and $q=r=2$, one can get
for $1<p<2$
    $$
    T:(L^2(w^{-p}), L^2)_{1/p',2}\times(L^2(w^{-p}), L^2)_{1/p',2}
    \rightarrow(\ell^{\beta_0}_{\infty}(\mathbb{C}),\ell^{\beta_1}_{\infty}(\mathbb{C}))_{2/p',1}
    $$
with the operator norm  $C\|w\|_{\mathfrak{L}^{\alpha,\beta,p}}$.
Finally, we use the following real interpolation space identities in Lemma \ref{id}
in order to obtain \eqref{bii}.
Indeed, from the lemma one can see that
$$(L^2(w^{-p}), L^2)_{1/p',2}=L^2(w^{-1})$$
for $1<p<2$, and
$$(\ell^{\beta_0}_{\infty}(\mathbb{C}),\ell^{\beta_1}_{\infty}(\mathbb{C}))_{2/p',1}=\ell_{1}^{0}(\mathbb{C})$$
if $(1-\frac2{p'})\beta_0+\frac2{p'}\beta_1=0$ (i.e., $\alpha + \beta =\frac{n+2}{2}$).
Hence, we get \eqref{bii} if $\alpha + \beta =\frac{n+2}{2}$ and $1<p<2$.

\begin{lem}\label{id}({\it cf. Theorems 5.4.1 and 5.6.1 in \cite{BL}})
Let $0<\theta<1$. Then one has
$$( L^{2} (w_0), L^{2} (w_1) )_{\theta,2} = L^2(w),\quad w= w_0^{1-\theta} w_1^{\theta},$$
and for $1\leq q_0,q_1,q\leq\infty$
$$(\ell^{s_0}_{q_0}, \ell^{s_1}_{q_1} )_{\theta, q}=\ell^s_q,\quad s= (1-\theta)s_0 + \theta s_1.$$
\end{lem}

When $\alpha + \beta >\frac{n+2}{2}$, we note that $\gamma:=\frac{p'}{2}(\alpha+\beta-\frac{n+2}{2})>0$.
Since $j\geq0$ and $\beta_1<0$, the estimates \eqref{222} and \eqref{233} are trivially satisfied
for $\beta_1$ replaced by $\beta_1-\gamma$.
Thus, by the same argument we only need to check that $(1-\frac2{p'})\beta_0+\frac2{p'}(\beta_1-\gamma)=0$
which is an easy computation.
Consequently, we get \eqref{bii} if $\alpha + \beta \geq \frac{n+2}{2}$ and $1<p<2$.
This completes the proof for \eqref{freq}.

\smallskip

It remains only to show the three estimates \eqref{21}, \eqref{22} and \eqref{23}.
For $j\geq0$, let $\{Q_\lambda\}_{\lambda\in 2^j\mathbb{Z}^{n+1}}$ be a collection of
cubes $Q_\lambda\subset\mathbb{R}^{n+1}$ centered at $\lambda$ with side length $2^j$.
Then by disjointness of cubes, it is easy to see that
\begin{align*}
\big|\big\langle(\psi_j K)\ast F,G\big\rangle\big|
&\leq\sum_{\lambda,\mu\in 2^j \mathbb{Z}^{n+1}}\big|\big\langle(\psi_j K)\ast(F\chi_{Q_\lambda}),G \chi_{Q_\mu} \big\rangle\big|\\
&\leq \sum_{\lambda\in 2^j \mathbb{Z}^{n+1}} \big| \big\langle (\psi_j K)\ast
(F\chi_{Q_\lambda}),G\chi_{\widetilde{Q}_\lambda}\big\rangle\big|,
\end{align*}
where $\widetilde{Q}_\lambda$ denotes the cube with side length $2^{j+2}$ and the same center as $Q_\lambda$.
By Young's and Cauchy-Schwartz inequalities, it follows now that
\begin{align}\label{33}
\big| \big\langle (\psi_j K) \ast F, G \big\rangle \big|
\nonumber&\leq \sum_{\lambda\in 2^j \mathbb{Z}^{n+1}}\|(\psi_j K)\ast(F\chi_{Q_\lambda})\|_{\infty}
\|G\chi_{\widetilde{Q}_\lambda}\|_1\\
\nonumber&\leq\sum_{\lambda\in 2^j \mathbb{Z}^{n+1}}\|\psi_jK\|_{\infty}\|F\chi_{Q_\lambda}\|_1
\|G\chi_{\widetilde{Q}_\lambda}\|_1\\
&\leq\|\psi_j K\|_{\infty}
\Big(\sum_{\lambda\in 2^j \mathbb{Z}^{n+1}}\|F\chi_{Q_\lambda}\|_1^2 \Big)^{\frac{1}{2}}
\Big(\sum_{\lambda\in2^j\mathbb{Z}^{n+1}}\|G\chi_{\widetilde{Q}_\lambda}\|_1^2\Big)^{\frac{1}{2}}.
\end{align}

To estimate the first term $\| \psi_j K \|_{\infty}$ in the right-hand side of \eqref{33},
we will use the following well-known lemma\footnote{\,It is essentially due to Littman \cite{L}. See also \cite{St}, VIII, Section 5, B.}, Lemma \ref{25}.
In fact, by applying the lemma with $\varphi(\xi)=|\xi|^a$, $a>1$, it follows that
    $$
    |K(x,t)|=\bigg|\int_{\mathbb{R}^n} e^{i(x\cdot\xi + t|\xi|^a)} \phi(|\xi|)^2 d\xi\bigg|
    \leq C(1+ |(x,t)| )^{-\frac{n}{2}},
    $$
since the Hessian matrix $H\psi$
has $n$ non-zero eigenvalues for each $\xi \in \{\xi\in\mathbb{R}^n: |\xi|\sim 1 \}$.
This implies the estimate
    \begin{equation}\label{stationary_phase}
    \|\psi_j K \|_{\infty}\leq C2^{-j \frac{n}{2}}.
    \end{equation}

\begin{lem}\label{25}
Let $H\varphi$ be the Hessian matrix given by $(\frac{\partial^2\varphi}{\partial\xi_i\partial\xi_j})$.
Suppose that $\eta$ is a compactly supported smooth function on $\mathbb{R}^n$
and $\varphi$ is a smooth function satisfying rank $H\varphi\geq k$ on the support of $\eta$.
Then, for $(x,t)\in \mathbb{R}^{n+1}$
$$\bigg|\int e^{i(x\cdot\xi+t\varphi(\xi))}\eta(\xi)d\xi\bigg|
\leq C(1+|(x,t)|)^{-\frac{k}{2}}.$$
\end{lem}

The other terms are estimated using H\"older's inequality, as follows:
    \begin{align}\label{get_Morrey}
    \sum_{\lambda\in 2^j \mathbb{Z}^{n+1}}\|F\chi_{Q_\lambda}\|_1^2
    \nonumber&=\sum_{\lambda\in 2^j \mathbb{Z}^{n+1}} \Big( \int_{Q_\lambda} |F\chi_{Q_\lambda}| w^{-\frac{p}{2}} w^{\frac{p}{2}}dxdt\Big)^2\\
    \nonumber&\leq\sum_{\lambda\in 2^j \mathbb{Z}^{n+1}} \Big( \int_{Q_\lambda} |F\chi_{Q_\lambda}|^2 w^{-p} dxdt \Big)\Big(\int_{Q_\lambda}w^{p}dxdt\Big)\\
    \nonumber&\leq \sup_{\lambda\in 2^j \mathbb{Z}^{n+1}}\Big(\int_{Q_\lambda}w^{p}dxdt\Big)
    \sum_{\lambda\in 2^j \mathbb{Z}^{n+1}} \Big( \int_{Q_\lambda} |F\chi_{Q_\lambda}|^2 w^{-p}dxdt\Big)\\
    &\leq C2^{j(n+1-(\alpha+\beta)p)}\|w\|_{\mathfrak{L}^{\alpha,\beta,p}}^p \| F \|_{L^2(w^{-p})}^2
    \end{align}
while
    \begin{align}\label{get_trivial}
    \sum_{\lambda\in 2^j\mathbb{Z}^{n+1}}\|F\chi_{Q_\lambda}\|_1^2
    \nonumber&\leq\sum_{\lambda\in 2^j\mathbb{Z}^{n+1}}\|F\chi_{Q_\lambda}\|_2^2\|\chi_{Q_\lambda}\|_2^2\\
    &\leq C2^{j(n+1)}\|F\|_2^2.
    \end{align}
Similarly for $\sum_{\lambda\in 2^j\mathbb{Z}^{n+1}}\|G\chi_{\widetilde{Q}_\lambda}\|_1^2$.

By combining \eqref{33}, \eqref{stationary_phase}, \eqref{get_Morrey} and \eqref{get_trivial},
one can easily get the desired three estimates \eqref{21}, \eqref{22} and \eqref{23}.

\medskip

Let us now turn to the second estimate \eqref{base} in the proposition.
We will show the following estimate
    \begin{equation}\label{eno}
    \bigg\|\int_{-\infty}^{t} e^{i(t-s)(-\Delta)^{a/2}} F(\cdot,s)ds\bigg\|_{L^2(w(x,t))}
    \leq C\|w\|_{\mathfrak{L}^{\alpha,\beta,p}} \|F\|_{L^2(w(x,t)^{-1})}
    \end{equation}
which implies \eqref{base}.
In fact, to obtain \eqref{base} from \eqref{eno},
first decompose the $L_t^2$ norm in the left-hand side of \eqref{base}
into two parts, $t\geq0$ and $t<0$. Then the latter can be reduced to the former
by changing the variable $t\mapsto-t$, and so it is only needed to consider the first part $t\geq0$.
But, since $[0,t)=(-\infty,t)\cap[0,\infty)$, applying \eqref{eno} with $F$ replaced by $\chi_{[0,\infty)}(s)F$,
one can bound the first part as desired.
To show \eqref{eno}, by duality we may show the following bilinear form estimate as before:
    \begin{equation*}
    \bigg|\bigg\langle\int_{-\infty}^{t} e^{i(t-s)(-\Delta)^{a/2}}P_0^2F(\cdot,s)ds,G(x,t)\bigg\rangle\bigg|
    \leq C\|w\|_{\mathfrak{L}^{\alpha,\beta,p}}\|F\|_{L^2(w^{-1})}\|G\|_{L^2(w^{-1})}.
    \end{equation*}
Let us write
    \begin{align*}
    \int_{-\infty}^{t}e^{i(t-s)(-\Delta)^{a/2}}P_0F(\cdot,s)ds
    &=\int_{\mathbb{R}}\chi_{(0,\infty)}(t-s)e^{i(t-s)(-\Delta)^{a/2}}P_0F(\cdot,s)ds\\
    &=K\ast F,
    \end{align*}
where
    $$
    K(x,t)=\int_{\mathbb{R}^n}\chi_{(0,\infty)}(t)e^{i(x\cdot\xi+t|\xi|^a)}\phi(|\xi|)d\xi.
    $$
Then, it is clear that the above bilinear estimate would follow from the same argument
for the homogeneous part \eqref{freq}.
So we omit the details.
\end{proof}


\end{document}